\documentclass[12pt,leqno]{article}
\usepackage{amssymb}
\usepackage{amsmath}
\usepackage{amsfonts}
\usepackage{amsthm}

\newtheorem{definition}{Definition}[section]
\newtheorem{note}{Note}[section]
\newtheorem{statement}{Statement}[section]
\newtheorem{theorem}{Theorem}[section]

\begin{document}
\title{Is Goldbach Conjecture true?}
\author{Paolo Starni}
\date{}

\maketitle

\begin{abstract}
We answer the question positively. In fact, we believe to have proved that every even integer $2N\geq3\times10^{6}$ is the sum of two odd distinct primes. Numerical calculations extend this result for $2N$ in the range $8-3\times10^{6}$. So, a fortiori, it is shown that {\em every even integer $2N>2$ is the sum of two primes (Goldbach conjecture)}. Of course, we would be grateful for comments and objections.
\end{abstract}

\section{Introduction}
Goldbach conjecture (1742) states that:
\begin{statement}[Goldbach form]
Every integer $N>5$ is the sum of three primes.
\end{statement}
An equivalent formulation due to Euler, called {\em strong} form, has replaced in literature the Goldbach form:
\begin{statement}[strong form]
\label{strong}
 Every even integer $2N>2$ is the sum of two primes.
 \end{statement}
Numerical calculations have verified it up to $4\times10^{18}$
 \cite{Oliveira}; for a remarkable theoretical result see the reference \cite{Chen}. The strong Golbach conjecture is also called {\em binary} or {\em even}. It implies the following weaker form (also called {\em odd} or {\em ternary}):
 \begin{statement}[weak form]
 Every odd integer $2N+1>5$ is the sum of three primes.
 \end{statement}
This formulation has been proved in the asymptotic case \cite{Vinogradov}, while a general proof \cite{Helfgott} is to the author knowledge still under consideration by the mathematical community. Here we consider a further formulation, implying the strong form, that we call {\em very strong}:
\begin{statement}[very strong form]
 Every even integer $2N>6$ is the sum of two odd distinct primes.
 \label{strongg}
 \end{statement}
Reference \cite{Oliveira} holds also for this very strong form (Oliveira e Silva, personal communication). Explicit experimental evidence of Statement \ref{strongg} up to $5\times10^{8}$ can be found in \cite{Richstein}: in fact, indicating with $r(2N)$ the number of Goldbach partitions of $2N$ (i.e. the number of unordered pairs of primes having sum equal to $2N$) it results $r(2N)>1$  in the range $4-5\times10^{8}$, excluding $r(4)=r(6)=r(8)=r(12)=1$.\\
We will show:
\begin{statement}[our result]
 Every integer $2N\geq3\times10^{6}$ is the sum of two odd distinct primes.
\end{statement}
Finally, Goldbach conjecture in its very strong form is shown combining our result with that found in \cite{Richstein} or\cite{Oliveira}.

\section{Preliminary remarks in order to prove the conjecture}
 Without explicit definitions all the numbers considered in what follows must be taken as strictly positive integers.
\begin{definition}
 Primes of type Q related to 2N (symbol: $Q_{j}(2N)$): are primes that divide 2N ($2=Q_{1}(2N)<Q_{2}(2N)<...<Q_{t}(2N)\leq N$).
\end{definition}
\begin{definition}
\label{primep}
Primes of type P related to 2N (symbol: $P_{j}(2N)$): are primes less than $2N-2$ and non-divisors of $2N$
($3\leq P_{1}(2N)<P_{2}(2N)<...<P_{h}(2N)<2N-2$). Their set is indicated as $P_{2N}$. The cardinality of the set  $P_{2N}$ is $card P_{2N}=h$. 
\end{definition}
\begin{note}
\label{partitionpq}
Primes of type P and Q relative to 2N are a partition of the set of the
primes less than $2N-2$.
\end{note}
\begin{definition}
\label{composite1}
Composites of type P related to 2N (symbol: $X_{j}(2N)$): are composites less than $2N-2,$ factorized into prime
factors only of type P related to 2N ($2N-2>X_{1}(2N)
>X_{2}(2N)>...>X_{s}(2N)=P_{1}^{2}(2N)$). Their set is indicated as $X_{2N}$ and $card X_{2N}=s$. 
\end{definition}
\begin{definition}
Integers of type P related to 2N (symbol: $a_{n}(2N)$): are primes or
composites of type P related to 2N ($2N-2>a_{1}(2N)>a_{2}(2N)>...>a_{w}(2N)=P_{1}(2N)$). Their set is indicated as $A_{2N}=X_{2N}\cup P_{2N}$ and $cardA_{2N}=s+h$.
\end{definition}
\begin{note}
In absence of ambiguity we will indicate, for example, $P_{j}$ instead of $P_{j}(2N)$.
\end{note}
\ \ Concerning the values of $card P_{2N}=h$ we give:
\begin{theorem}
\label{bertrand}
If $2N>6$, then $h\geq2.$
\end{theorem}
\begin{proof}
 The strongest formulation of Bertrand's postulate \cite[p. 373]{Hardy},
 states that: for every $N>3$ there exists an odd prime $P_{r} $
satisfying $N<P_{r}<2N-2$. It is remarkable that, for Definition \ref{primep}, $P_{r}$
is a prime of type $P$. Besides, $2N-P_{r}=a_{r}$ is an integer of type $P$;
otherwise a prime of type $Q,$ let it be $Q_{v}$, divides $a_{r} $ and so
$Q_{v}\mid P_{r}$,\ i.e. $Q_{v}=P_{r}$, in contradiction with Note  \ref{partitionpq}. Since
$a_{r}<N<P_{r}$, there is at least a prime of type $P$ different from $P_{r}$
and so $h\geq2$.
\end{proof}
\begin{theorem}
\label{euler}
If $2N\geq3\times10^{6}$, then $card X_{2N}>card P_{2N}$.  
\end{theorem}
\begin{proof}
We consider the following two relations concerning the functions $\pi(x)$ (number of primes $\leq x$) and $\phi(x)$ (totient Euler's function):
\begin{description}
\item (i) $\pi(x)<1.25506\frac{x}{lnx}$ for $x>1$ \cite[ p. 233, theorem 8.8.1]{Bach}
\item (ii) $\phi(x)>\frac{x}{e^{\gamma}lnlnx+\frac{3}{lnlnx}}$ for $x\geq3$   \cite[ p. 234, theorem 8.8.7]{Bach}; $\gamma=0.577...$ is the Euler-Mascheroni constant.
\end{description}
The function $\phi(x)$ counts the number of the positive integers less than $x$ and prime to $x$. In this way $cardA_{2N}=\phi(2N)-2$, because $\phi(2N)$ considers also $2N-1$ and $1$; but these numbers are not integers of type P.
We have
\begin{description}
\item (iii) $cardP_{2N}=h<\pi(2N)$
\item (iv)  $cardX_{2N}=s=cardA_{2N}-cardP_{2N}=\phi(2N)-2-h$.
\end{description}
Therefore, from (i)-(iv) it results
\begin{equation}
s-h=\phi(2N)-2-2h>f(2N)
\label{eq:1}
\end{equation}
where
\begin{equation*}
f(2N)=\frac{2N}{1.781lnln2N+\frac{3}{lnln2N}}-2-2.510\frac{2N}{ln2N}.
\end{equation*}
$f(2N)$ is a divergent sequence; by numerical computations it is increasing for $2N>10^{6}$  and for $2N\geq3\times10^{6}$ its values are greater than $10^{3}$. In this way, a fortiori, from \eqref{eq:1} it follows the proof. 
\end{proof}
We introduce now an essential concept for our purposes:
\begin{definition}
 G-system related to $2N$: it is the system
 \begin{equation}
\left\{
\begin{array}
[c]{c}
\hspace{0.9cm}\Gamma_{1}:2N-P_{1}=a_{n_{1}}=a_{1}\\
\Gamma_{2}:2N-P_{2}=a_{n_{2}}\\
\vdots\\
\Gamma_{j}:2N-P_{j}=a_{n_{j}}\\
\vdots\\
\hspace{1.25cm}\Gamma_{h-1}:2N-P_{h-1}=a_{n_{h-1}}\\
\hspace{0.2cm}\Gamma_{h}:2N-P_{h}=a_{n_{h}}
\end{array}
\right.
\label{eq:2}
\end{equation}
\end{definition}
Fixed in \eqref{eq:2} $2N>6$, we remark that:
\begin{note}
Existence of the G-system is guaranteed by Theorem \ref{bertrand}.
\end{note}
\begin{note}
$(a_{n_j})_{j=1,2,...,h}$ is the subsequence of the sequence $(a_{n})_{n=1,2,...,h+s}$ that does not contain the terms generated by $2N-X_{s-r}, r=0,1,...,s-1$; $a_{1}=a_{n_{1}}$ because $2N-P_{1}$ is the greatest integer of type $P$, while, for example, if $X_{s}=P_{1}^{2}<P_{2}$, then $a_{2}>a_{n_{2}}$.
\end{note}
\begin{note}
The h equations are not necessarily distinct; in fact, if
$a_{n_{j}}=P_{k}$, then $\Gamma_{j}$ is equivalent to $\Gamma_{k}$ (and
$2N=P_{j}+P_{k}$).
\end{note}
\begin{note}
\label{pnondivisorofa}
$P_{j}\nmid a_{n_{j}},\forall j$; otherwise $P_{j}\mid2N$
in contradiction with Definition  \ref{primep}.
\end{note}
\begin{theorem}
\label{gsystemcomposite}
Each term of $(a_{n_j})_{j=1,2,...,h}$ is  $composite\Leftrightarrow2N$ is not the sum of two odd
distinct primes.
\end{theorem}
\begin{proof}
Immediate.
\end{proof}
\begin{note}
Theorem   \ref{gsystemcomposite} holds also for N prime; in fact, from Note
 \ref{pnondivisorofa}, $P_{j}\neq a_{n_{j}} \forall j$ and so the equation $2N-P_{j}=P_{j}$ does not belong to
the G-system (in fact, in this case, $P_{j}$ would not be a prime of type P). 
\end{note}
\section{Proof of the conjecture}
\begin{theorem}
Every even integer $2N\geq3\times10^{6}$ is the sum of two odd distinct primes.
\label{main}
\end{theorem}
\begin{proof}
Let us suppose that $2N\geq3\times10^{6}$ is not the sum of two odd distinct primes. From Theorem  \ref{gsystemcomposite}
it follows that each term of $(a_{n_j})_{ j=1,2,...,h}$ is composite; so, in particular, the first relation at the top of system \eqref{eq:2}, being $a_{n_{1}}=a_{1}=X_{1}$, is $2N-X_1=P_1$.\\ Let us suppose $2N-X_{2}>P_{2}$. Thus:
\begin{equation}
\left\{
\begin{array}
[c]{c}
2N-X_{1}= P_{1}\\
\hspace{-0.3cm}2N-\alpha = P_{2}\\
2N-X_{2}> P_{2}\\
\vdots
\end{array}
\right.  
\label{eq:3} 
\end{equation}
Since $X_{2}<\alpha<  X_{1}$, $\alpha$ is a prime (of type P) and this is impossible because $2N$ does not verify the conjecture. So \eqref{eq:3} becomes 
\begin{equation}
\left\{
\begin{array}
[c]{c}
2N-X_{1}= P_{1}\\
2N-X_{2}\leq P_{2}\\
\vdots
\end{array}
\right.   
\end{equation}
Proceeding in analogous way the system \eqref{eq:2} may be written as
\begin{equation}
\left\{
\begin{array}
[c]{c}
2N-X_{1}= P_{1}\\
2N-X_{2}\leq P_{2}\\
\vdots\\
2N-X_{j}\leq P_{j}\\
\vdots\\
\hspace{0.8cm}2N-X_{h-1}\leq P_{h-1}\\
2N-X_{h}\leq P_{h}
\end{array}
\right. 
\label{eq:5} 
\end{equation}
Starting now from the bottom of the system \eqref{eq:2} we have
\begin{equation}
\left\{
\begin{array}
[c]{c}
\vdots\\
\hspace{1.05cm}2N-  P_{h-1}=a_{n_{h-1}}\geq X_{s-1}\\
2N-P_{h}=a_{n_{h}}\geq X_{s}
\end{array}
\right. 
\label{eq:6} 
\end{equation}
It occurs because $a_{n_{h}}$ is a composite of type P and $X_{s}$  is the smallest composite of the same type (and with similar consideration $a_{n_{h-1}}\geq X_{s-1}$). So system \eqref{eq:2} may be rewritten as
\begin{equation}
\left\{
\begin{array}
[c]{c}
\hspace{0.7cm}2N- X_{1}\geq  P_{h-s+1}\\
\vdots\\
\hspace{0.7cm}2N- X_{s-j}\geq  P_{h-j}\\
\vdots\\
\hspace{0.7cm}2N-  X_{s-1}\geq P_{h-1}\\
2N-X_{s}\geq P_{h}
\end{array}
\right. 
\label{eq:7} 
\end{equation}
Considering $s-j=1$ we obtain the relation at the top of the system \eqref{eq:7}. Comparing this relation with that at the top of system \eqref{eq:5}, we obtain  $P_{1}\geq P_{h-s+1}$. Since $P_{1}$ is the smallest prime of type $P$, we have $P_{1}=P_{h-s+1}$ and, therefore, $1=h-s+1$. Thus  $cardX_{2N}=cardP_{2N}$  (see Definitions \ref{primep} and \ref{composite1}) and this, by Theorem \ref{euler}, is impossible. In this way it follows the proof. 
\end{proof}
At this point we obtain the aforementioned result:
\begin{theorem}
Every even integer $2N>6$ is the sum of two odd distinct primes.
\end{theorem}
\begin{proof}
It follows immediately from Theorem \ref{main} and \cite{Richstein} or \cite{Oliveira}.
\end{proof}
\vspace{1.0cm}



\end{document}